\documentclass[11pt]{article}
\usepackage{epsfig}
\usepackage{authblk}
\usepackage{epstopdf}
\usepackage{graphics}
\usepackage{amsmath}
\usepackage{amsthm}
\usepackage{amssymb}
\usepackage{array}
\usepackage{subfig} 
\usepackage[font={footnotesize}]{caption}
\usepackage{relsize}
\usepackage[nottoc]{tocbibind}
\usepackage{hyperref}
\usepackage{float}

\usepackage{chngcntr}
\usepackage{wrapfig}

\counterwithin*{equation}{section}
  
\newtheorem{thm}{Theorem}[section] 
\newtheorem{lem}[thm]{Lemma}  
\newtheorem{cor}[thm]{Corollary} 
\newtheorem{prop}[thm]{Proposition}  

\newtheorem{hyp}{Hypothesis}

\newtheorem{defn}[thm]{Definition}

\newtheoremstyle{named}{}{}{\itshape}{}{\bfseries}{.}{.5em}{\thmnote{#3's }#1}
\theoremstyle{named}

\theoremstyle{definition}
\newtheorem{rmk}{Remark}
\newtheorem{example}{Example}

\title{Ultracontractive Properties for Directed Graph Semigroups with Applications to Coupled Oscillators}

\author{Jason J. Bramburger\\ Division of Applied Mathematics\\ Brown University \\ Providence, Rhode Island 02912\\ USA}

\date{} 

\begin{document} 

\maketitle

\abstract{It is now well known that ultracontractive properties of semigroups with infinitesimal generator given by an undirected graph Laplacian operator can be obtained through an understanding of the geometry of the underlying infinite weighted graph. The aim of this work is to extend these results to semigroups with infinitesimal generator given by a directed graph Laplacian operator, through an analogous inspection of the geometry of the underlying directed graph. In particular, we introduce appropriate nomenclature to discuss the geometry of an infinite directed graph, as well as provide sufficient conditions to extend ultracontractive properties of undirected graph Laplacians to those of the directed variety. Such directed graph Laplacians can often be observed in the study of coupled oscillators, where recent work in \cite{Bramburger} made explicit the link between synchronous patterns to systems of identically coupled oscillators and ultracontractive properties of undirected graph semigroups. Therefore, in this work we demonstrate the applicability of our results on directed graph semigroups by extending the aforementioned investigation beyond the idealized case of identically coupled oscillators.}

\section{Introduction}\label{sec:Introduction} 

The study of heat equations on graphs have long been a topic of inquiry which has successfully related the geometric properties of the underlying graph to time-dependent estimates of the behaviour of the semigroup generated by the associated graph Laplacian \cite{Delmotte,Frank,Grigoryan1,Grigoryan2}. A discrete heat equation takes the form of the linear ordinary differential equation
\begin{equation} \label{graphODE}
	\dot{x}_v(t) = \sum_{v'\in V}w(v,v')(x_{v'}(t) - x_v(t)),
\end{equation}
for each $v \in V$. Here $\dot{x}_v(t)$ denotes the derivative of $x_v(t)$ with respect to the independent variable $t$, $V$ is the countably infinite vertex set of an underlying graph, and $w(v,v')$ represents the weight of the edge from vertex $v$ to vertex $v'$.\footnote{This will be made more precise in the following section.} In the case of undirected (or symmetric) graphs, much work has been done to connect the behaviour of a random walk on the underlying graph to the long-time dynamics of solutions to the differential equation (\ref{graphODE}) associated to the graph \cite{Bramburger,Delmotte,Folz,TelcsBook}. This work has successfully introduced ultracontractive properties into the study of heat kernels on symmetric graphs, thus continuing a long investigation into decay of one-parameter semigroups which dates back at least to the seminal work of Varopoloulus \cite{Varopoulus}.

It appears that the study of discrete heat equations on graphs is greatly skewed towards undirected graphs, with few results pertaining to ultracontractive properties of system (\ref{graphODE}) associated with a directed graph. Therefore, it is the intention of this manuscript to introduce a set of sufficient conditions which allow one to obtain ultracontractive properties of the semigroup generated by the linear operator governing the right-hand side of (\ref{graphODE}), based upon the geometry of the underlying directed graph. Precisely, in this manuscript a set of sufficient conditions is provided for graphs of dimension two and up which can guarantee that the ultracontractive properties from undirected graphs can be extended to the general setting of directed graphs. These ultracontractive properties are equivalent to the uniform decay in $t$ of solutions to (\ref{graphODE}) over various Banach spaces of real sequences indexed by the vertex set $V$.      

Aside from their connection with random walks, discrete heat equations of the form (\ref{graphODE}) arise naturally in the study of coupled oscillators, where the stability of a synchronous state is often understood via the geometry of an associated graph. Although this connection has been well-studied in the finite-dimensional setting \cite{Chen,ErmentroutStability}, there still remain a number of open problems pertaining to the infinite-dimensional setting. Recent work has initiated the investigation into the connection between graph geometry and stability in infinite systems of coupled oscillators by restricting the investigation to identically coupled oscillators \cite{Bramburger}. This restriction to identically coupled oscillators lacks the generality that is already well understood in the finite-dimensional setting, and therefore in this manuscript we aim to describe how our work on system (\ref{graphODE}) can be used to extend the results of \cite{Bramburger} beyond such an idealized scenario. Therefore, our work herein leads to a more robust result detailing sufficient conditions for the stability of infinitely-many coupled oscillators. 

Systems of the form (\ref{graphODE}) have also been documented in the study of the stability of traveling wave solutions to lattice dynamical systems \cite{Hoffman}. This investigation required a tedious analysis using the Fourier transform to obtain decaying bounds on an associated Green's function, which was then used to infer linearized stability of an associated ordinary differential equation. It is therefore the intention of this work to provide a framework in which future investigations into the stability of solutions to lattice dynamical systems can readily obtain linearized stability through a careful checking of the conditions on (\ref{graphODE}) laid out in this manuscript, potentially reducing the amount of difficulty required to obtain decaying bounds on a Green's function. Hence, it has become a long-term goal to apply the results of this work to the diverse and expanding study of stability in lattice dynamical systems. 

This manuscript is organized as follows. In Section~\ref{sec:Preliminaries} we introduce the proper nomenclature, notation, and hypotheses to discuss discrete heat equations on graphs, as well as introduce some new notation to properly analyze directed graphs. Then, in Section~\ref{sec:Symmetric} we discuss some of the known results for undirected graphs, as well as provide some necessary extensions of this work which will become useful when discussing directed graphs. Our main result is Theorem~\ref{thm:AsymDecay} in Section~\ref{sec:Asymmetric} which provides a set of sufficient conditions on the geometry of a directed graph to obtain uniform decay of solutions to (\ref{graphODE}). Section~\ref{sec:Advection} is dedicated to demonstrating the importance of Hypothesis~\ref{hyp:AsymWeights}, which forms the major assumption on the geometry of the directed graphs considered in this manuscript. An example of a system of the form (\ref{graphODE}) is provided for which this assumption fails and it is shown that the decay of solutions cannot be understood via the methods outlined in this manuscript. Finally, in Section~\ref{sec:CoupledOscillators} we connect these results to the stability of coupled oscillators, resulting in Theorem~\ref{thm:PhaseStability}, which is supplemented by a brief discussion of an application of this theorem.

\section{Definitions and Hypotheses} \label{sec:Preliminaries} 

We consider a graph $G = (V,E)$ with a countably infinite collection of vertices, $V$, and a set of oriented edges between these vertices, $E$. If there exists an edge $e\in E$ originating at vertex $v$ and terminating at vertex $v'$ then we will write $v\sim v'$, but we note that since the edges are assumed to be oriented the relation $v\sim v'$ is not necessarily symmetric. Furthermore, we may equivalently consider the set of edges $E$ as a subset of the product $V \times V$ by writing $\{v,v'\} \in E$ if there exists an edge originating at vertex $v$ and terminating at vertex $v'$. A graph is called {\em strongly connected} (or simply connected in this manuscript) if for any two vertices $v,v'\in V$ there exists a finite sequence of vertices in $V$, $\{v_1,v_2, \dots , v_n\}$, such that $v\sim v_1$, $v_1\sim v_2$, $\dots$, $v_n \sim v'$. We will only consider connected graphs for the duration of this work.

We will also consider a weight function on the edges between vertices, written $w:V\times V \to \mathbb{R}$, such that for all $v,v'\in V$ we have $w(v,v') \neq 0$ if and only if $v \sim v'$. This then leads to the notion of a weighted oriented graph, written as the triple $G = (V,E,w)$. We emphasize that $w$ is not necessarily symmetric with respect to its arguments, even in the case when $v\sim v'$ and $v'\sim v$ for some $v,v'\in V$. Moreover, it should be noted that in the interest of full generality we have not assumed that the weights are nonnegative, but only that all edges must have a nonzero weight. The weight function further allows us to consider the graph Laplacian (sometimes combinatorial graph Laplacian) associated to the graph $G = (V,E,w)$ given by the linear operator, $L$, acting on the real sequences $x = \{x_v\}_{v\in V}$ by
\begin{equation}\label{graphLap}
	[Lx]_v = \sum_{v'\in V}w(v,v')(x_{v'} - x_v),
\end{equation}
so that (\ref{graphODE}) can be written abstractly as the linear ordinary differential equation $\dot{x} = Lx$, upon suppressing the dependence on the independent variable $t$ for convenience. Hence, the general solution to (\ref{graphODE}) with initial condition $x_0$ can be written $x(t) = e^{Lt}x_0$, where $e^{Lt}$ is the semigroup with infinitesimal generator $L$. It will therefore be our goal in this manuscript to obtain ultracontractive properties on the semigroup $e^{Lt}$, which are equivalent to determining uniform decay properties of the solution $x(t)$.   

Natural spatial settings for the graph Laplacian operator are the real sequence spaces 
\begin{equation} \label{ellpSpace}
	\ell^p(V) = \bigg\{x= \{x_v\}_{v\in V} |\ \sum_{v\in V} |x_v|^p < \infty\bigg\},	
\end{equation}  
for any $p\in [1,\infty)$. The vector space $\ell^p(V)$ becomes a Banach space when equipped with the norm
\begin{equation} 
	\|x\|_p := \bigg(\sum_{v\in V} |x_v|^p\bigg)^{\frac{1}{p}}.
\end{equation} 
We may also consider the Banach space $\ell^\infty(V)$, the vector space of all uniformly bounded real sequences indexed by $V$ with norm given by
\begin{equation} \label{ellInfty}
	\|x\|_\infty := \sup_{v\in V} |x_v|.
\end{equation}
It should be noted that these definitions extend to any countable index set $V$, independent of a respective graph. 

The potential asymmetry of the edges and weights on the graph $G$ make the direct application of results for graphs with undirected edges unlikely, and therefore we wish to develop a method of extending these results to the setting of (\ref{graphODE}) for a general directed graph. Let us begin by defining the function $w_\mathrm{sym}:V\times V \to \mathbb{R}$ by
\begin{equation}
	w_\mathrm{sym}(v,v') := \frac{w(v,v') + w(v',v)}{2},
\end{equation}  
so that $w_\mathrm{sym}(v,v') = w_\mathrm{sym}(v',v)$ for all $v,v'\in V$. Similarly, we will define the function $w_\mathrm{skew}:V\times V \to \mathbb{R}$ by
\begin{equation}
	w_\mathrm{skew}(v,v') := \frac{w(v,v') - w(v',v)}{2},
\end{equation}  
so that $w_\mathrm{skew}(v,v') = -w_\mathrm{skew}(v',v)$ for all $v,v'\in V$. Hence, one sees that 
\[
	w(v,v') = w_\mathrm{sym}(v,v') + w_\mathrm{skew}(v,v'),  
\]
for all $v,v'\in V$. This leads to the following definition.

\begin{defn} \label{def:Laplacians} 
The graph Laplacian (\ref{graphLap}) induces the linear operators $L_\mathrm{sym}$ and $L_\mathrm{skew}$ given by
	\begin{equation} \label{Lsym} 
		[L_\mathrm{sym}x]_v = \sum_{v'\in V}w_\mathrm{sym}(v,v')(x_{v'} - x_v),		
	\end{equation}
	and
	\begin{equation} \label{Lsym} 
		[L_\mathrm{skew}x]_v = \sum_{v'\in V}w_\mathrm{skew}(v,v')(x_{v'} - x_v).		
	\end{equation}
	We refer to $L_\mathrm{sym}$ as the {\bf symmetric graph Laplacian} induced by $L$, and $L_\mathrm{skew}$ as the {\bf skew-symmetric graph Laplacian} induced by $L$.
\end{defn}

It should immediately be noted that $L = L_\mathrm{sym} + L_\mathrm{skew}$. Moreover, the function $w_\mathrm{sym}$ and the linear operator $L_\mathrm{sym}$ also leads to the definition of an underlying undirected weighted graph.

\begin{defn} \label{def:SymmetricGraph} 
	The {\bf symmetric graph} induced by $G$, denoted $G_\mathrm{sym}$, is the graph with vertex set $V$ and edge set, $E_\mathrm{sym}$, defined by assigning an undirected edge connecting $v,v' \in V$ if and only if $w_\mathrm{sym}(v,v') \neq 0$. 
\end{defn} 
    
The graph $G_\mathrm{sym}$ becomes a weighted graph when considered with the symmetric weight function $w_\mathrm{sym}$. Therefore, it will be through the graph $G_\mathrm{sym} = (V,E_\mathrm{sym},w_\mathrm{sym})$ and the associated symmetric graph Laplacian $L_\mathrm{sym}$ that we will work to understand decay properties of the linear equation (\ref{graphODE}). We present the following hypothesis which is fundamental to our interpretation of $G_\mathrm{sym}$, and in turn $G$.

\begin{hyp}\label{hyp:SymmetricWeights} 
	The graph $G_\mathrm{sym} = (V,E_\mathrm{sym},w_\mathrm{sym})$ satisfies the following:
	\begin{enumerate}
		\item The function $w_\mathrm{sym}:V\times V \to \mathbb{R}$ is nonnegative, and furthermore if $w(v,v')\cdot w(v',v)\neq 0$, then $w_\mathrm{sym}(v,v') > 0$.
		\item There exists a constant $M > 0$ such that $w_\mathrm{sym}(v,v') \leq M$ for all $v,v' \in V$.
		\item The set $N(v) := \{v'\in V:\ w_\mathrm{sym}(v,v') > 0\}$ is such that there exists a constant $D \geq 1$ such that $1 \leq |N(v)| \leq D$ for all $v \in V$.
	\end{enumerate}	 
\end{hyp}

Hypothesis~\ref{hyp:SymmetricWeights} first details that all edge weights of $G_\mathrm{sym}$ are strictly positive. We should note that this does not contradict our assumption that the weight function $w$ associated with the original directed graph $G$ can assume negative values. Indeed, we simply have imposed that if $w(v,v') < 0$, then we necessarily have $w(v',v) > 0$ and $w(v,v') + w(v',v) > 0$. This also leads to our second assumption in Hypothesis~\ref{hyp:SymmetricWeights} which details that the graph $G_\mathrm{sym}$ essentially takes all edges in $G$, makes them unoriented, and assigns a weight which is the average of the directed edge weights between each pair of vertices. Hence, we have assumed that the creation of $G_\mathrm{sym}$ does not disconnect two vertices, and since $G$ was assumed to be connected, we therefore have that $G_\mathrm{sym}$ is also connected. 

Hypothesis~\ref{hyp:SymmetricWeights} further details that we assume the edge weights to be uniformly bounded above, and that each vertex in $G_\mathrm{sym}$ is connected to a finite number of vertices. When a graph exhibits this latter property, it is often said to be {\em locally finite}. It should be noted that our assumption is slightly more restrictive than just being locally finite though, as we have assumed that the number of vertices each vertex is connected to is uniformly bounded from above. The set $N(v)$ represents the neighbourhood of $v \in V$ in the graph $G_\mathrm{sym}$. In the context of the directed graph $G$, under Hypothesis~\ref{hyp:SymmetricWeights} we have that $N(v)$ represents the set of all vertices $v' \in V$ for which $v \sim v'$ or $v' \sim v$. It will be through Hypothesis~\ref{hyp:SymmetricWeights} that we will work to understand the decay of solutions to the linear equation (\ref{graphODE}).

We now turn to the skew-symmetric graph Laplacian $L_\mathrm{skew}$ and the associated weight function $w_\mathrm{skew}$. Let us define the quantity
\begin{equation} \label{WValue}
	W := \sum_{v\in V}\sum_{v' \in V} |w_\mathrm{skew}(v,v')| \in [0,\infty].
\end{equation}
This leads to the following hypothesis.

\begin{hyp} \label{hyp:AsymWeights}
	The quantity $W$ defined in (\ref{WValue}) is finite.	
\end{hyp}

We remark that we make no assumption on the exact magnitude of $W$, we simply assume that it is finite. This implies that for any $\varepsilon > 0$, the weights $w(v,v')$ and $w(v',v)$ will be $\varepsilon$-close for infinitely many $v,v' \in V$. In Section~\ref{sec:Advection} we demonstrate that in the absence of Hypothesis~\ref{hyp:AsymWeights} solutions to (\ref{graphODE}) cannot necessarily be understood through the geometry of the associated symmetric graph. We conclude this section with the following simple example which illustrates all of the definitions and hypotheses put forth in this section.

\begin{example}\label{Ex:Graph}
	Consider a directed graph with index set $V = \mathbb{Z}$ with edges from vertices $n$ to $n+1$ and vice-versa, except that there is no edge from vertices indexed by $0$ to $1$. The edge weights are given by
	\[
		\begin{split}
			w(n,n+1) &= 1 - \frac{1}{1 + n^2}, \\
			w(n+1,n) &= 1 + \frac{1}{1 + n^2},
		\end{split}
	\]
	for all $n \in \mathbb{Z}$, where we note that $w(0,1) = 0$ meaning that there is no edge from $0$ to $1$. Figure~\ref{fig:ExampleGraph} provides a visual representation of this graph. The important point here is that for large $|n|$ we have that the weights $w(n,n+1)$ and $w(n+1,n)$ become uniformly close together at a rate of $\mathcal{O}(n^{-2})$, which will guarantee that Hypothesis~\ref{hyp:AsymWeights} is indeed satisfied. 
	
	\begin{figure} 
	\centering
		\includegraphics[width = 0.6\textwidth]{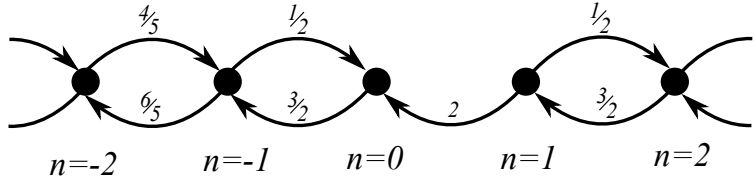}
		\caption{The graph discussed in Example~\ref{Ex:Graph}. The vertex set is given by the set of integers with edges from vertices $n$ to $n+1$ and vice-versa, with the exception of an edge from $0$ to $1$. Edge weights are given above each edge.}
		\label{fig:ExampleGraph}	
\end{figure} 
	
	Then, using the definition of $w_\mathrm{sym}$ and $w_\mathrm{skew}$ above we get that 
	\[
		\begin{split}
			w_\mathrm{sym}(n,n+1) &= 1, \\
			w_\mathrm{skew}(n,n+1) &= \frac{-1}{1+n^2},
		\end{split}
	\]  
	for all $n \in \mathbb{Z}$, along with the symmetry conditions $w_\mathrm{sym}(n+1,n) = w_\mathrm{sym}(n,n+1)$ and $w_\mathrm{skew}(n+1,n) = -w_\mathrm{skew}(n,n+1)$.  The graph $G_\mathrm{sym}(\mathbb{Z},E_\mathrm{sym},w_\mathrm{sym})$ is simply the standard one-dimensional integer lattice, where successive integers are connected by an undirected edge of weight $1$. Hence, it is very easy to check that Hypothesis~\ref{hyp:SymmetricWeights} does indeed hold for $G_\mathrm{sym}$. Moreover, the quantity $W$ in this case is given by
	\[
		W = \sum_{n = -\infty}^\infty \frac{2}{1 + n^2} = 2\pi\coth(\pi),
	\]
	where $\coth$ is the hyperbolic cotangent function. Hence, $W < \infty$ in this case, and hence Hypothesis~\ref{hyp:AsymWeights} holds for this graph as well.  
\end{example} 

\section{Ultracontractive Properties for Undirected Graphs} \label{sec:Symmetric} 

In this section we provide a review of the relevant results for undirected graphs. We will see that an understanding of the geometry of the graph $G_\mathrm{sym} = (V,E_\mathrm{sym},w_\mathrm{sym})$ can be used to obtain uniform decay of solutions to the linear ordinary differential equation 
\begin{equation} \label{symmetricODE}
	\dot{x} = L_\mathrm{sym}x.
\end{equation}
It should be noted that much of the information in this section comes as a review of the work in \cite{Bramburger}, and many of the graph theoretic facts and definitions can be found in, for example, $\cite{Delmotte,TelcsBook}$. Hence, in this section we provide assumptions that lead to the algebraic decay of solutions to (\ref{symmetricODE}), which will be utilized in the following section to obtain algebraic decay of solutions to (\ref{graphODE}).   

We begin by defining the {\em measure} of a vertex of the graph $G_\mathrm{sym}$, written $m:V \to [0,\infty]$, and defined by
\begin{equation} \label{VertexMeasure}
	m(v) := \sum_{v'\in V} w_\mathrm{sym}(v,v') = \sum_{v \in N(v)} w_\mathrm{sym}(v,v').
\end{equation} 
We note that Hypothesis~\ref{hyp:SymmetricWeights} dictates that $0 < m(v) \leq MD$, for all $v \in V$, and hence the measure $m$ is well-defined. This notion extends to the volume of a subset, $V_0 \subset V$, by defining 
\begin{equation}
	{\rm Vol}(V_0) := \sum_{v \in V_0} m(v).
\end{equation} 
Hence, we see that $G_\mathrm{sym}$ can be interpreted as a measure space with $\sigma$-algebra given by the power set of $V$.

Connected undirected graphs also have a natural metric associated to them, here denoted $\rho$, which returns the smallest number of edges needed to traverse from one vertex to another. This metric allows for the consideration of a ball of radius $r \geq 0$ centred at the vertex $v \in V$, denoted by
\begin{equation}
	B(v,r) := \{v'\ |\ \rho(v,v') \leq r\}.
\end{equation} 
For simplicity we will simply write ${\rm Vol}(v,r)$ to denote ${\rm Vol}(B(v,r))$. The combination of the graph metric and the vertex measure allows one to interpret a weighted graph as a {\em metric-measure space}.

We now provide a series of definitions to describe the geometry of $G_\mathrm{sym}$.

\begin{defn} \label{def:VolumeGrowth} 
	The weighted graph $G_\mathrm{sym} = (V,E_\mathrm{sym},w_\mathrm{sym})$ satisfies a {\bf uniform polynomial volume growth} condition of order $d$, abbreviated VG(d), if there exists $d > 0$ and $c_{vol,1},c_{vol,2} > 0$ such that
	\begin{equation}
		c_{vol,1}r^d \leq {\rm Vol}(v,r) \leq c_{vol,2}r^d,
	\end{equation}
	for all $v \in V$ and $r \geq 0$.
\end{defn}

The value $d$ in Definition~\ref{def:VolumeGrowth} is often referred to as the {\em dimension} of the graph $G_\mathrm{sym}$. A potential reason for this is that the characteristic examples of graphs satisfying VG(d) are the integer lattices $\mathbb{Z}^d$ with an edge between two vertices $n,n' \in \mathbb{Z}^d$ if and only if $\|n - n'\|_1 = 1$, and all edge weights taken to be identically $1$ \cite{BarlowCoulhonGrigoryan}. In Example~\ref{Ex:Graph} we saw that the resulting symmetric graph is exactly of this type, and therefore it satisfies VG(1). It should be noted that $d$ need not be an integer, as one may construct fractal graphs which satisfy VG(d) for non-integer valued $d > 0$. For the duration of this work we will restrict our attention to $d \geq 2$, since the methods of Section~\ref{sec:Asymmetric} fail when $d < 2$. 

\begin{defn} \label{def:Delta} 
	We say $G_\mathrm{sym} = (V,E_\mathrm{sym},w_\mathrm{sym})$ satisfies the {\bf local elliptic property}, denoted $\Delta$, if there exists an $\alpha > 0$ such that
	\begin{equation}
		w_\mathrm{sym}(v,v') \geq \alpha m(v)
	\end{equation}
	for all $v \in V$ and $v' \in N(v)$.
\end{defn}

It was pointed out in \cite[Lemma~$3.5$]{Bramburger} that a sufficient condition to satisfy this local elliptic property, $\Delta$, is to have the edge weights, $w_\mathrm{sym}(v,v')$, bounded above and below by positive constants for all $v \in V$, $v' \in N(v)$ and to have a uniform upper bound on the number of elements in $N(v)$ over $v \in V$. Of course, Hypothesis~\ref{hyp:SymmetricWeights} takes care of two thirds of these sufficient conditions, but in the interest of generality we refrain from assuming the third and final condition, as these conditions were not found to be necessary. 

\begin{defn} 
	The weighted graph $G_\mathrm{sym} = (V,E_\mathrm{sym},w_\mathrm{sym})$ satisfies the {\bf Poincar\'e inequality}, abbreviated PI, if there exists a constant $C_{PI} > 0$ such that
	\begin{equation}
		\sum_{v \in B(v_0,r)} m(v)|x_v - x_B(v_0)|^2 \leq C_{PI} r^2\bigg( \sum_{v,v' \in B(v_0,2r)} w_\mathrm{sym}(v,v')(x_v - x_{v'})^2\bigg),
	\end{equation}	
	for all real sequences $\{x_v\}_{v\in V}$, all $v_0 \in V$, and all $r > 0$, where
	\begin{equation}
		x_B(v_0) = \frac{1}{{\rm Vol}(v_0,r)} \sum_{v\in B(v_0,r)} m(v)x_v.
	\end{equation}
\end{defn}

It is immediately apparent that proving an undirected graph satisfies the Poincar\'e inequality is a significant analytical undertaking. Some methods were outlined in \cite{PoincareInequalities}, and in \cite{Bramburger} the notion of a rough isometry was introduced to demonstrate that a graph satisfies PI. We refrain from going into further detail here, but direct the reader to those sources for a full analytical treatment of the Poincar\'e inequality with regards to undirected graphs. 

It is well-known \cite{Horn,KellerLenz,Weber} that $L_\mathrm{sym}$ is the infinitesimal generator of the semigroup $P_t = e^{L_\mathrm{sym}t}$. Moreover, the semigroup $P_t$ acts on the real sequences $x = \{x_v\}_{v\in V}$ by 
\begin{equation}\label{KernelExpansion}
	[P_tx]_v = \sum_{v'\in V} p_t(v,v')x_{v'},
\end{equation}
where $p_t(v,v')$ are transition probabilities generated by a random walk on the weighted graph $G_\mathrm{sym}$ \cite{Delmotte,Horn,Weber} (see \cite{Bramburger} for complete details). This leads to the following proposition, which summarizes the work of \cite[Section~$3.3$]{Bramburger}. 

\begin{prop}[\cite{Bramburger}, \S3.3]\label{prop:SymDecay} 
	Assume that $G_\mathrm{sym} = (V,E_\mathrm{sym},w_\mathrm{sym})$ satisfies Hypothesis~\ref{hyp:SymmetricWeights}, $\Delta$, PI, and VG(d) for some $d > 0$. Then, there exists a constant $C_\mathrm{sym} > 0$ such that
	\begin{equation}
		\begin{split}
			\|P_tx\|_p &\leq C_\mathrm{sym}\|x\|_p, \\
			\|P_tx\|_p &\leq C_\mathrm{sym}(1 + t)^{-\frac{d}{2}(1 - \frac{1}{p})}\|x\|_1, 
		\end{split}
	\end{equation}	
	for all $p \in [1,\infty]$ and real sequences $x = \{x_v\}_{v\in V} \in \ell^1(V)$.
\end{prop}

We now provide an extension of Proposition~\ref{prop:SymDecay}, which will be integral to our work here. Let us begin by defining the functions
\[
	Q_p(x) = \bigg(\sum_{v\in V} \sum_{v'\in N(v)} |x_{v'} - x_v|^p\bigg)^\frac{1}{p},
\]
for all $p \geq 1$ and
\[
	Q_\infty(x) = \sup_{v\in V, v' \in N(v)} |x_{v'} - x_v|.
\]
It is easy to see that for each $p \in [1,\infty]$ the functions $Q_p$ satisfy $Q_p(x) \leq 2D\|x\|_p$ for all $x \in \ell^p$, and are semi-norms on $\ell^p(V)$ for each $p \geq 1$. Furthermore, since $G_\mathrm{sym}$ is assumed to be connected, it follows that the $Q_p$ vanish if and only if $x$ is a constant sequence. The components $|x_{v'} - x_v|$ are typically interpreted as the discrete analogue of a directional derivative of the sequence $x$ in the direction of the edge $\{v,v'\}$. Hence, $Q_p$ can be thought to be the $p$-norm of the (discrete) gradient of the sequences in $\ell^p(V)$. We now provide the following lemma which extends the bounds of Proposition~\ref{prop:SymDecay}.

\begin{lem}\label{lem:GradDecay} 
	Assume that $G_\mathrm{sym} = (V,E_\mathrm{sym},w_\mathrm{sym})$ satisfies Hypothesis~\ref{hyp:SymmetricWeights}, $\Delta$, PI, and VG(d) for some $d > 0$. There exists constants $C_Q,\eta > 0$ such that for all $x = \{x_v\}_{v \in V} \in \ell^1(V)$ we have
	\[
		Q_p(P_tx) \leq C_Q(1 + t)^{-\frac{d}{2}(1-\frac{1}{p})-\eta}\|x\|_1, 
	\]
	for all $t \geq 0$ and $p \in [1,\infty]$. 
\end{lem}    

\begin{proof}
	It was shown by Delmotte that any graph satisfying $\Delta$, PI, and VG(d) for some $d > 0$, must also satisfy a Parabolic Harnack Inequality \cite{Delmotte}, which we do not explicitly state here because it will not be necessary to our result. But, the work of \cite[Theorem~$2.32$]{GyryaSaloffCoste} dictates that any graph (or more generally metric space) satisfying the Parabolic Harnack Inequality further satisfies the estimate
\[
	|p_t(v_1,v_3) - p_t(v_2,v_3)| \leq C_0m(v_3) \bigg(\frac{\rho(v_1,v_2)}{\sqrt{1 + t}}\bigg)^\beta p_{2t}(v_1,v_3) 
\]
for all $v_1,v_2,v_3 \in V$ and some independent constants $C_0, \beta > 0$. Hence, assuming Hypothesis~\ref{hyp:SymmetricWeights}, for all $v, v'' \in V$ and $v' \in N(v)$ we have
\[
	|p_t(v,v'') - p_t(v',v'')| \leq C_0M(1 + t)^{-\frac{\beta}{2}}p_{2t}(v,v''),	
\]	
since $\rho(v,v') = 1$ because $v' \in N(v)$.

Then, using the form for $P_t$ given in (\ref{KernelExpansion}), for all $x \in \ell^1(V)$, $v\in V$, and $v'\in N(v)$ we have 
\[
	\begin{split}
		|[P_tx]_{v'} - [P_tx]_v| &\leq \sum_{v'' \in V} |p_t(v,v'') - p_t(v',v'')| |x_{v''}| \\
			&\leq C_0M(1 + t)^{-\frac{\beta}{2}} \sum_{v'' \in V} p_{2t}(v,v'') |x_{v''}| \\
			&= C_0M(1 + t)^{-\frac{\beta}{2}}[P_{2t}|x|]_v. 
	\end{split}
\]
This therefore implies that
\[
	\begin{split}
		Q_p(P_tx) &= \bigg(\sum_{v \in V}\sum_{v' \in N(v)} |[P_tx]_{v'} - [P_tx]_v|\bigg)^\frac{1}{p} \\
		&\leq C_0M(1 + t)^{-\frac{\beta}{2}}\bigg(\sum_{v \in V}\sum_{v' \in N(v)} [P_{2t}|x|]_v\bigg)^\frac{1}{p} \\
		&\leq C_0D^\frac{1}{p}M(1 + t)^{-\frac{\beta}{2}}\bigg(\sum_{v \in V}[P_{2t}|x|]_v\bigg)^\frac{1}{p} \\
		&\leq C_0D^\frac{1}{p}M(1 + t)^{-\frac{\beta}{2}}\|P_{2t}|x|\|_p,
	\end{split}
\]
for all $p \in [1,\infty)$ since $|N(v)| \leq D$ for all $v \in V$. Then, from Proposition~\ref{prop:SymDecay} we have that 
	\[
		\|P_{2t}|x|\|_p \leq C_\mathrm{sym}\|P_{t}|x|\|_p \leq C_\mathrm{sym}^2(1 + t)^{-\frac{d}{2}(1 - \frac{1}{p})}\|x\|_1,
	\] 
	for some constant $C_\mathrm{sym} > 0$, where we have introduced the notation $|x| = \{|x_v|\}_{v \in V}$. This then gives that 
	\[
		Q_p(P_tx) \leq C_0C_\mathrm{sym}^2D^\frac{1}{p}M(1 + t)^{-\frac{d}{2}(1 - \frac{1}{p})-\frac{\beta}{2}}\|x\|_1,	
	\]
	which proves the cases $p \in [1,\infty)$ with $\eta = \frac{\beta}{2}$. The case $p = \infty$ follows in a nearly identical fashion, and is therefore omitted. This completes the proof. 
\end{proof} 

\section{Ultracontractive Properties for Directed Graphs} \label{sec:Asymmetric} 

In this section we show that an understanding of the graph $G_\mathrm{sym}$ and its associated symmetric graph Laplacian $L_\mathrm{sym}$ can be used to understand the decay of solutions to (\ref{graphODE}). The following theorem is our main result on the decay of solutions to (\ref{graphODE}), and its proof will be broken up into a series of lemmas throughout this section.

\begin{thm} \label{thm:AsymDecay} 
	Consider the linear ordinary differential equation (\ref{graphODE}), and construct $G_\mathrm{sym}$ and $L_\mathrm{sym}$ as defined in Definition~\ref{def:Laplacians} and \ref{def:SymmetricGraph}, respectively. Assume that $G_\mathrm{sym}$ satisfies Hypothesis~\ref{hyp:SymmetricWeights}, $\Delta$, PI, and VG(d) for some $d \geq 2$ and that Hypothesis~\ref{hyp:AsymWeights} is true. Then, there exists a continuous, positive, strictly increasing function $f:[0,\infty) \to (0,\infty)$ and a constant $\eta > 0$ such that for all $x_0 \in \ell^1$ we have that the solution $x(t) = e^{Lt}x_0$ to (\ref{graphODE}) satisfies the following decay estimates:
	\begin{equation}
		\begin{split}
			\|x(t)\|_p &\leq f(W)(1+t)^{-\frac{d}{2}(1 - \frac{1}{p})}\|x_0\|_1, \\
			Q_p(x(t)) &\leq f(W)(1+t)^{-\frac{d}{2}(1 - \frac{1}{p}) - \eta}\|x_0\|_1,
		\end{split}
	\end{equation}    
	for all $t \geq 0$ and $p \in [1,\infty]$.
\end{thm}

\begin{rmk}
We note that our results only pertain to those graphs $G_\mathrm{sym}$ which satisfy VG(d) with $d \geq 2$, i.e. at least two-dimensional graphs. Of course this is a minor shortcoming of Theorem~\ref{thm:AsymDecay}, but we will see in the following proofs that the case $d < 2$ (particularly $d = 1$ for many applications) presents a major technical hurdle which cannot be overcome with the methods put forth in this manuscript. This same technical hurdle was encountered in the previous results of \cite{Bramburger}, and therefore it would be interesting if alternative methods were proposed which overcome the restriction to $d \geq 2$. 
\end{rmk}

\begin{rmk}
	All of our analysis in this manuscript relies heavily on the definitions of $w_\mathrm{sym}$ and $w_\mathrm{skew}$. Therefore, it would be interesting in the future to explore different definitions for these weights to see if the results of Theorem~\ref{thm:AsymDecay} can be extended to an even wider range of directed graphs than those considered herein.
\end{rmk}

We now proceed with the proof of Theorem~\ref{thm:AsymDecay}, beginning with the following lemma.

\begin{lem} \label{lem:AsymBnd} 
	Assume Hypothesis~\ref{hyp:AsymWeights}. Then, for all $x\in\ell^\infty(V)$ we have
	\[
		\|L_\mathrm{skew}x\|_1 \leq WQ_\infty(x),
	\]
	where $W < \infty$ is the quantity defined in (\ref{WValue}).
\end{lem}

\begin{proof}
	We begin by remarking that the assumption $x \in \ell^\infty(V)$ is merely to guarantee that $Q_\infty(x)$ is finite and may be loosened under appropriate conditions. Then, using $L_\mathrm{skew}$ given in Definition~\ref{def:Laplacians} we have that 
	\[
		|[L_\mathrm{skew}x]_v| \leq \sum_{v'\in V} |w_\mathrm{skew}(v,v')||x_{v'} - x_v| \leq \bigg(\sum_{v'\in V} |w_\mathrm{skew}(v,v')|\bigg)Q_\infty(x),
	\]
	for every $v \in V$. Then, taking the sum over all $v \in V$ we arrive at
	\[
		\|L_\mathrm{skew}x\|_1 \leq \bigg(\sum_{v \in V}\sum_{v'\in V} |w_\mathrm{skew}(v,v')|\bigg)Q_\infty(x) = WQ_\infty(x),	
	\]
	which proves the lemma.
\end{proof} 

Now, if $x(t)$ is a solution to (\ref{graphODE}) with initial condition $x(0) = x_0 \in \ell^1$, we trivially have that
\[
	\dot{x}(t) = L_\mathrm{sym}x(t) + L_\mathrm{skew}x(t).
\]
Then, using the variation of constants formula we obtain the equivalent integral form of the ordinary differential equation (\ref{graphODE}), given as
\begin{equation}\label{IntForm}
	x(t) = P_tx_0 + \int_0^t P_{t-s}L_\mathrm{skew}x(s)ds,
\end{equation}
where $P_t = e^{L_\mathrm{sym}t}$ is the semigroup with infinitesimal generator $L_\mathrm{sym}$ described in the previous section. Moreover, since we have assumed that $G_\mathrm{sym}$ satisfies Hypothesis~\ref{hyp:SymmetricWeights}, $\Delta$, PI, and VG(d) for some $d \geq 2$, we obtain the decay properties of both Proposition~\ref{prop:SymDecay} and Lemma~\ref{lem:GradDecay}. We now use the integral form (\ref{IntForm}) to prove Theorem~\ref{thm:AsymDecay}, but first we provide a useful lemma from $\cite{IntegralLemma}$.

\begin{lem}[{\em \cite{IntegralLemma}, \S 3, Lemma 3.2}] \label{lem:IntegralLemma} 
	Let $\gamma_1, \gamma_2$ be positive real numbers. If $\gamma_1,\gamma_2 \neq 1$ or if $\gamma_1 = 1 < \gamma_2$ then there exists a $C_{\gamma_1,\gamma_2} > 0$ such that
	\begin{equation}
		\int_0^t (1 + t - s)^{- \gamma_1}(1 + s)^{-\gamma_2}ds \leq C_{\gamma_1,\gamma_2} (1 + t)^{-\min\{\gamma_1 + \gamma_2 - 1, \gamma_1, \gamma_2\}},
	\end{equation}  	 
	for all $t \geq 0$.
\end{lem} 

\begin{lem}\label{lem:GradDecay2} 
	Assume that $G_\mathrm{sym} = (V,E_\mathrm{sym},w_\mathrm{sym})$ satisfies Hypothesis~\ref{hyp:SymmetricWeights}, $\Delta$, PI, and VG(d) for some $d \geq 2$, and that Hypothesis~\ref{hyp:AsymWeights} is true. Then, there exists a continuous, positive, strictly increasing function $f_1:[0,\infty) \to (0,\infty)$ such that for all $x_0 \in \ell^1$, the solution $x(t)$ to (\ref{graphODE}) with $x(0) = x_0$ satisfies
	\[
		Q_\infty(x(t)) \leq f_1(W)(1 + t)^{-\frac{d}{2}-\eta}\|x_0\|_1,
	\]
	for all $t \geq 0$, where $\eta > 0$ is the constant guaranteed by Lemma~\ref{lem:GradDecay}.
\end{lem}

\begin{proof}
	Through straightforward manipulations of the integral form (\ref{IntForm}) one obtains 
	\[
		Q_\infty(x(t)) \leq Q_\infty(P_tx_0) + \int_0^t Q_\infty(P_{t-s}L_\mathrm{skew}x(s))ds.
	\]
	Then, using Lemmas~\ref{lem:GradDecay} and \ref{lem:AsymBnd} we obtain
	\[
		\begin{split}
			Q_\infty(x(t)) &\leq C_Q(1 + t)^{-\frac{d}{2}-\eta}\|x_0\|_1 + C_Q\int_0^t (1 + t -s)^{-\frac{d}{2}-\eta}\|L_\mathrm{skew}x(s))\|_1ds \\
			&\leq C_Q(1 + t)^{-\frac{d}{2}-\eta}\|x_0\|_1 + C_QW\int_0^t (1 + t -s)^{-\frac{d}{2}-\eta}Q_\infty(x(s))ds,  
		\end{split}		
	\]
	where $C_Q > 0$ is the constant guaranteed by Lemma~\ref{lem:GradDecay}. We now apply Gronwall's Inequality to obtain
	\[
		\begin{split}
		Q_\infty(x(t)) \leq &C_Q(1 + t)^{-\frac{d}{2}-\eta}\|x_0\|_1 \\ 
		&+ C_Q^2W\|x_0\|_1\int_0^t (1 + t -s)^{-\frac{d}{2}-\eta}(1 + s)^{-\frac{d}{2}-\eta}e^{C_QW\int_s^t(1 + t -r)^{-\frac{d}{2}-\eta}dr}ds. 	
		\end{split}
	\]
	
	Now, since $d \geq 2$, we have that $\frac{d}{2} + \eta > 1$, and hence 
	\[
		e^{C_QW\int_s^t(1 + t -r)^{-\frac{d}{2}-\eta}dr} \leq e^\frac{2C_QW}{d+2\eta-2}. 
	\] 
	for all $s,t \geq 0$. Then, combining this bound with the result of Lemma~\ref{lem:IntegralLemma} we find that
	\[
		\int_0^t (1 + t -s)^{-\frac{d}{2}-\eta}(1 + s)^{-\frac{d}{2}-\eta}e^{CW\int_s^t(1 + t -r)^{-\frac{d}{2}-\eta}dr}ds \leq C_{\frac{d}{2}+\eta,\frac{d}{2}+\eta}e^\frac{2CW}{d+2\eta-2}(1 + t)^{-\frac{d}{2}-\eta}. 	
	\]
	Putting this all together therefore gives
	\[
		Q_\infty(x(t)) \leq C_Q(1 + C_QWC_{\frac{d}{2}+\eta,\frac{d}{2}+\eta}e^\frac{2CW}{d+2\eta-2})(1 + t)^{-\frac{d}{2}-\eta}\|x_0\|_1,	
	\]
	which allows one to define $f_1(W) = C_Q(1 + C_QWC_{\frac{d}{2}+\eta,\frac{d}{2}+\eta}e^\frac{2CW}{d+2\eta-2})$, thus completing the proof.
\end{proof} 

\begin{cor} \label{cor:1InfDecay} 
	Assume that $G_\mathrm{sym} = (V,E_\mathrm{sym},w_\mathrm{sym})$ satisfies Hypothesis~\ref{hyp:SymmetricWeights}, $\Delta$, PI, and VG(d) for some $d \geq 2$, and that Hypothesis~\ref{hyp:AsymWeights} is true. Then, there exists a continuous, positive, strictly increasing function $f_2:[0,\infty) \to (0,\infty)$ such that for all $x_0 \in \ell^1$, the solution $x(t)$ to (\ref{graphODE}) with $x(0) = x_0$ satisfies
	\[
		\begin{split}
			\|x(t)\|_1 &\leq f_2(W)\|x_0\|_1,\\
			\|x(t)\|_\infty &\leq f_2(W)(1 + t)^{-\frac{d}{2}}\|x_0\|_1,
		\end{split}
	\]
	for all $t \geq 0$.	
\end{cor}

\begin{proof}
	This proof follows in a similar way to that of Lemma~\ref{lem:GradDecay2}. Beginning with the $\ell^1(V)$ bound, we use (\ref{IntForm}) and the bounds from Proposition~\ref{prop:SymDecay} to see that 
	\[
		\begin{split}
			\|x(t)\|_1 &\leq \|P_tx_0\|_1 + \int_0^t \|P_{t-s}L_\mathrm{skew}x(s)\|_1ds \\
			&\leq C_\mathrm{sym}\|x_0\|_1 + C_\mathrm{sym}\int_0^t \|L_\mathrm{skew}x(s)\|_1ds \\
			&\leq C_\mathrm{sym}\|x_0\|_1 + C_\mathrm{sym}W\int_0^t Q_\infty(x(s))ds.
		\end{split}
	\] 
	Then, using Lemma~\ref{lem:GradDecay2} we obtain
	\[
		\begin{split}
			\|x(t)\|_1 &\leq C_\mathrm{sym}\|x_0\|_1 + C_\mathrm{sym}f_1(W)W\|x_0\|_1\int_0^t (1 + s)^{-\frac{d}{2}-\eta}ds \\
			&\leq C_\mathrm{sym}\bigg(1 + \frac{2f_1(W)W}{d + 2\eta-2}\bigg)\|x_0\|_1,
		\end{split}
	\] 
	for all $t\geq 0$ since $\frac{d}{2}+\eta > 1$, which proves the first bound.
	
	Through a nearly identical manipulation to that of Lemma~\ref{lem:GradDecay2} we arrive at
	\[
		\begin{split}
		\|x(t)\|_\infty &\leq C_\mathrm{sym}(1 + t)^{-\frac{d}{2}}\|x_0\|_1 + C_\mathrm{sym}W\int_0^t (1 + t-s)^{-\frac{d}{2}}Q_\infty(x(s))ds \\
			&\leq C_\mathrm{sym}(1 + t)^{-\frac{d}{2}}\|x_0\|_1 + C_\mathrm{sym}f_1(W)W\|x_0\|_1\int_0^t(1 + t-s)^{-\frac{d}{2}}(1 + s)^{-\frac{d}{2}-\eta}ds \\
			&\leq C_\mathrm{sym}\bigg(1 + C_{\frac{d}{2},\frac{d}{2}+\eta}f_1(W)W\bigg)(1 + t)^{-\frac{d}{2}}\|x_0\|_1,   
		\end{split}
	\]
	by Lemma~\ref{lem:IntegralLemma}. Hence, we may define $f_2:[0,\infty) \to (0,\infty)$ by
	\[
		f_2(W) := C_\mathrm{sym}\max\bigg\{1 + \frac{2f_1(W)W}{d + \eta},1 + C_{\frac{d}{2},\frac{d}{2}+\eta}f_1(W)W\bigg\},
	\]
	which proves the lemma.
\end{proof} 

\begin{cor}\label{cor:pDecay} 
	Assume that $G_\mathrm{sym} = (V,E_\mathrm{sym},w_\mathrm{sym})$ satisfies Hypothesis~\ref{hyp:SymmetricWeights}, $\Delta$, PI, and VG(d) for some $d \geq 2$, and that Hypothesis~\ref{hyp:AsymWeights} is true. Then, for all $x_0 \in \ell^1$, the solution $x(t)$ to (\ref{graphODE}) with $x(0) = x_0$ satisfies
	\[
		\|x(t)\|_p  \leq f_2(W)(1 + t)^{-\frac{d}{2}(1-\frac{1}{p})}\|x_0\|_1,
	\]
	for all $t \geq 0$ and $p \in [1,\infty]$, where $f_2:[0,\infty) \to (0,\infty)$ is the function from Corollary~\ref{cor:1InfDecay}.	
\end{cor}

\begin{proof}
	This proof is a straightforward application of the log-convexity property of the $\ell^p$ norms, which dictates that for any $1 \leq p_0 \leq p_1 \leq \infty$ for all $x \in \ell^{p_0}(V)$ we have
	\begin{equation}
		\|x\|_{q} \leq \|x\|_{p_0}^{1 - \gamma}\|x\|_{p_1}^\gamma,
	\end{equation} 	
	where $q$ is defined by
	\[
		\frac{1}{q} = \frac{1 - \gamma}{p_0} + \frac{\gamma}{p_1}
	\] 
	for every $0 < \gamma < 1$. The proof is obtained by taking $p_0 = 1$ and $p_1 = \infty$ and applying the bounds from Corollary~\ref{cor:1InfDecay}.
\end{proof} 

Corollaries~\ref{cor:1InfDecay} and \ref{cor:pDecay} therefore give the proof of the $p$-norm bounds on the solution stated in Theorem~\ref{thm:AsymDecay}. The remaining $Q_p$ follow in exactly the same way to those of the $p$-norms, and are therefore omitted. This completes the proof of Theorem~\ref{thm:AsymDecay}.

\section{Importance of Hypothesis~\ref{hyp:AsymWeights}} \label{sec:Advection} 

We now detail a situation in which Hypothesis~\ref{hyp:AsymWeights} fails, and show that in this scenario we cannot obtain the decay rates of Theorem~\ref{thm:AsymDecay}. Since our results only apply to graphs of dimension two or higher, we will work with a two dimensional graph, although we note simpler examples can be created for one-dimensional graphs.

Let us consider the vertex set $V = \mathbb{Z}^2$ along with the linear ordinary differential equation
\begin{equation} \label{PureAdvection}
	\dot{x}_{i,j} = \begin{cases} 
      (x_{i-1,j} - x_{i,j}) + (x_{i,j-1} - x_{i,j}) & j\geq 1 \\
      (x_{i-1,j} - x_{i,j}) & j = 0 \\
      (x_{i-1,j} - x_{i,j}) + (x_{i,j+1} - x_{i,j}) & j \leq -1 
   \end{cases} 
\end{equation}
In the context of our present work, we find that the associated graph is composed of directed edges connecting $(i,j)$ to $(i-1,j)$, along with directed edges connecting $(i,j)$ to $(i,j-1)$ when $j \geq 1$ and $(i,j)$ to $(i,j+1)$ when $j \leq -1$, all with identical weights of $1$. Furthermore, the associated symmetric graph then has an edge set for which every $(i,j) \in \mathbb{Z}^2$ is connected to $(i\pm1,j)$ and $(i,j\pm1)$, with identical weights of $\frac{1}{2}$. Figure~\ref{fig:DirectedGraph} provides a visualization of this directed graph. Importantly, one may follow the methods of \cite[Section~6]{Bramburger} to see that this associated symmetric graph indeed satisfies $\Delta$, PI, and VG(2). Then, the associated symmetric graph leads to expected decay bounds of the order $(1 + t)^{-1}$ for all $t \geq 0$. But we note that the only requirement that fails to apply Theorem~\ref{thm:AsymDecay} is the condition that $W < \infty$, required by Hypothesis~\ref{hyp:AsymWeights}. 

\begin{figure} 
	\centering
		\includegraphics[width = 0.45\textwidth]{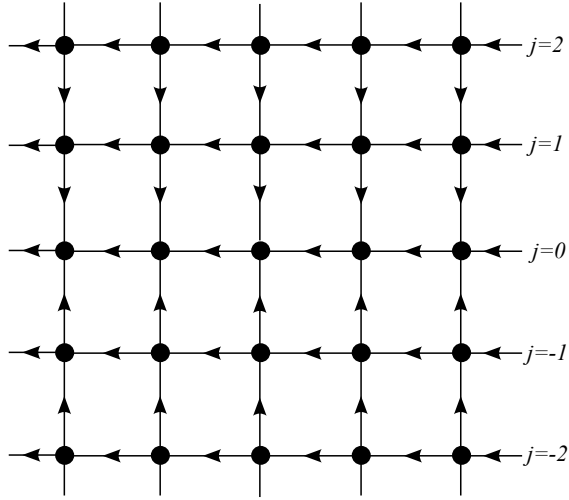}
		\caption{The directed graph associated to the linear differential system (\ref{PureAdvection}). Here the vertices lie in one-to-one correspondence with the elements of $(i,j) \in\mathbb{Z}^2$ and the direction of the edges is given by the arrow. All edges have weight exactly $1$. The resulting symmetric graph looks nearly identical, but with the arrows removed from the edges. In the case of the associated symmetric graph all edges have weight $\frac{1}{2}$.}
		\label{fig:DirectedGraph}	
\end{figure} 

Now, let us take the initial condition $x_0 = \{x_{i,j}^0\}_{(i,j)\in\mathbb{Z}^2}$ given by 
\[
	x_{i,j}^0 = \left\{
     		\begin{array}{cl} 1 & (i,j) = (0,0) \\ \\
		0 & (i,j) \neq (0,0).
		\end{array}
   	\right.
\] 
It is a straightforward argument to find that with this initial condition we have $x_{i,0}(t) = 0$ for all $i < 0$ and $t \geq 0$. The reason for this is that each element with index $j = 0$ depends only on those elements to the left of them, and since only the site $i = j = 0$ is activated with this initial condition, it can only influence those elements with $j = 0$ to the right of it.

We begin by observing that at index $(i,j) = (0,0)$ we have
\[
	\dot{x}_{0,0} = -x_{0,0} \implies x_{0,0}(t) = \mathrm{e}^{-t},
\] 
since $x_{-1,0}(t) = 0$ for all $t \geq 0$ and $x_{i,j}(0) = 1$. Then, moving to index $(i,j) = (1,0)$ we can substitute the solution for $x_{0,0}(t)$ to obtain
\[
	\dot{x}_{1,0} = e^{-t} -x_{1,0} \implies x_{1,0}(t) = t\mathrm{e}^{-t},
\] 
since $x_{1,0}(0) = 0$. Continuing in this way, an inductive argument shows that 
\begin{equation}\label{LDSSoln}
	x_{i,0}(t) = \frac{t^i}{i!}\mathrm{e}^{-t},
\end{equation}
for all $t \geq 0$. We plot the first few of these functions of visual reference in Figure~\ref{fig:LDSSoln}.

\begin{figure} 
	\centering
		\includegraphics[width = 0.8\textwidth]{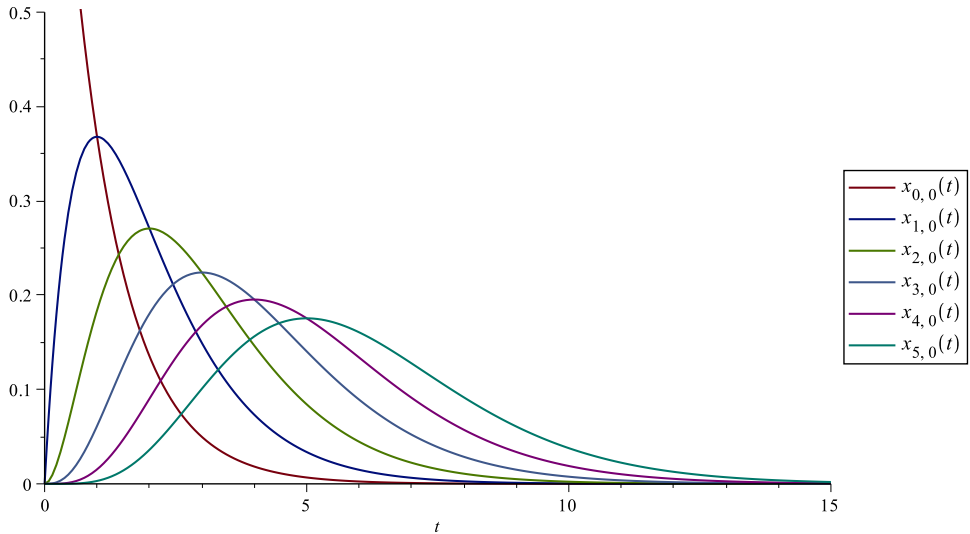}
		\caption{The functions (\ref{LDSSoln}) for $i = 0,1,2,3,4,5$. Note the unique global maximum of $x_{i,0}(t)$ at $t = i$.}
		\label{fig:LDSSoln}	
\end{figure} 

Then, for each $i \geq 1$, differentiating $x_{i,0}(t)$ with respect to $t$ gives
\[
	\frac{dx_{i,0}}{dt}(t) = \frac{t^{i-1}}{(i-1)!}\bigg(1-\frac{t}{i}\bigg)\mathrm{e}^{-t}. 
\] 
Hence, $x_{i,0}(t)$ attains its global maximum at $t = i$, and this maximum is given by
\[
	x_{i,0}(i) = \frac{i^i}{i!}\mathrm{e}^{-i}.
\]
Using Stirling's Approximation we find that $x_{i,0}(i) \geq \mathrm{e}^{-1}i^{-\frac{1}{2}}$. Hence, we see that the solution $x(t)$ to the differential equation (\ref{PureAdvection}) is such that
\[
	\|x(i)\|_\infty \geq \mathrm{e}^{-1}i^{-\frac{1}{2}}, 
\]
which shows that there cannot exists a constant $C > 0$ such that $\|x(t)\|_\infty \leq C(1 + t)^{-1}$ for all $t\geq 0$ since $i^{-\frac{1}{2}}$ cannot be bounded uniformly by a constant multiple of the function $(1 + i)^{-1}$ for all $i \geq 0$. Therefore, system (\ref{PureAdvection}) provides an example of a system for which the failure to have Hypothesis~\ref{hyp:AsymWeights}, but all other hypotheses of Theorem~\ref{thm:AsymDecay} hold, leads to solutions which can only decay at a rate of $(1 + t)^{-\frac{1}{2}}$, significantly slower than the $(1 + t)^{-1}$ that holds for the given associated symmetric graph. This presents a major problem in the analysis of the following section, since we require decays rates of at least $(1 + t)^{-1}$ to apply bootstrapping arguments to extend from linear ordinary differential equations to local asymptotical stability of nonlinear ordinary differential equation. Hence, in this case an understanding of the associated symmetric graph cannot inform our understanding of the directed graph and the decay of solutions to the differential equation (\ref{PureAdvection}).

\section{Application to Coupled Oscillators} \label{sec:CoupledOscillators} 

We now reserve this final section for an application of the results of Theorem~\ref{thm:AsymDecay}. To begin, it is well-known that systems of weakly coupled oscillators can be reduced through a process of averaging to a single phase variable under minor technical assumptions \cite{Corinto,ErmentroutAverage,WeaklyConnectedBook,Udeigwe}. In complete generality for a countable index set $V$ these systems take the form 
\begin{equation}\label{PhaseLattice}
	\dot{\theta}_v = \omega_v + \sum_{v' \in V\setminus\{v\}} H(\theta_{v'} - \theta_v,v,v'),
\end{equation}
where the function $H:\mathbb{R} \times V \times V \to \mathbb{R}$ is assumed to be smooth and $2\pi$-periodic in the first variable. The constants $\omega_v \in \mathbb{R}$ are taken to represent intrinsic differences in the oscillators and/or external inputs. In the work of \cite{Bramburger} it was assumed that the functions $H$ were independent of $(v,v')$, leading to the limited focus on identically coupled oscillators. With the results of the previous section, we are now able to expand to more general functions $H$, thus providing a more robust result to that of \cite{Bramburger}.

A solution to (\ref{PhaseLattice}) is called {\em phase-locked} (or synchronous) if it takes the form 
\begin{equation}\label{PhaseAnsatz}
	\theta_v(t) = \Omega t + \bar{\theta}_v,
\end{equation}
where $\bar{\theta} = \{\bar{\theta}_v\}_{v \in V}$ are time-independent phase-lags and the elements $\theta_v(t)$ are moving with identical velocity $\Omega\in\mathbb{R}$. Assuming the existence of a phase-locked solution to (\ref{PhaseLattice}) of the form (\ref{PhaseAnsatz}), the resulting linearization about this solutions leads to the linear operator, denoted $L_{\bar{\theta}}$, acting on the real sequences $x = \{x_v\}_{v\in V}$ by
\begin{equation}\label{PhaseLinear}
	[L_{\bar{\theta}}x]_v = \sum_{v'\in V\setminus\{v\}} H'(\bar{\theta}_{v'}-\bar{\theta}_v,v,v')(x_{v'} - x_v),
\end{equation}
for all $v \in V$, where the prime notation denotes differentiation with respect to the first component of $H$. The form of $L_{\bar{\theta}}$ given in (\ref{PhaseLinear}) should be immediately recognized as of the form of a graph Laplacian operator with 
\[
	w(v,v') = H'(\bar{\theta}_{v'}-\bar{\theta}_v,v,v'), 
\] 
for all $v,v' \in V$. This leads to the nontrivial extension of \cite[Theorem~4.5]{Bramburger}.

\begin{thm}\label{thm:PhaseStability} 
	Consider the system (\ref{PhaseLattice}) for a twice-differentiable function $H:\mathbb{R} \times V \times V \to \mathbb{R}$ such that the derivatives with respect to the first component are uniformly bounded in $\mathbb{R}\times V \times V$, and assume this system of equations possesses a phase-locked solution of the form (\ref{PhaseAnsatz}), denoted $\theta^{lock}(t)$. Then, if the resulting linear operator $L_{\bar{\theta}}$ defined in (\ref{PhaseLinear}) satisfies the hypotheses of Theorem~\ref{thm:AsymDecay}, we have the following: there exists an $\varepsilon > 0$ for which every $\theta_0 = \{\theta_{v,0}\}_{v \in V}$ with the property that
	\begin{equation}
		\|\theta_0 - \bar{\theta}\|_1 \leq \varepsilon,	
	\end{equation}   
	leads to a unique solution of (\ref{PhaseLattice}), $\theta(t)$ for all $t \geq 0$, satisfying the following properties:
	\begin{enumerate}
		\item $\theta(0) = \theta_0$.
		\item $\theta(t) - \theta^{lock}(t) \in \ell^p(V)$ for all $p \in [1,\infty]$.
		\item There exists a $C > 0$ such that 
			\begin{equation}
				\|\theta(t) - \theta_v^{lock}(t)\|_p \leq C (1 + t)^{-\frac{d}{2}(1 - \frac{1}{p})}\|\theta_0 - \bar{\theta}\|_1,
			\end{equation}
		for all $t \geq 0$ and $p \in [1,\infty]$.
	\end{enumerate}  
\end{thm} 

Due to the results of Theorem~\ref{thm:AsymDecay}, the proof of Theorem~\ref{thm:PhaseStability} is identical to the proof of \cite[Theorem~4.5]{Bramburger} and is therefore omitted. Prior to concluding this section, we comment on a simple application of Theorem~\ref{thm:PhaseStability} to optimally convey these results. Consider system (\ref{PhaseLattice}) with $\omega_v = \omega \in \mathbb{R}$ for all $v \in V$, and 
\[
	H(x,v,v') = k_{v,v'}\sin(x),
\]
where $K = [k_{v,v'}]_{v,v'\in V}$ is an infinite matrix of coupling coefficients. We note that no assumption on the signs of the $k_{v,v'}$ will be made. A trivial example of a phase-locked solution to such a system of coupled oscillators is obtained by taking $\Omega = \omega$ and $\bar{\theta}_v = 0$ for all $v \in V$. Hence, linearizing about this phase-locked solution results in a linear operator of the form on the right-hand side of (\ref{graphODE}) with 
\[
	w(v,v') = k_{v,v'},
\]  
for all $v,v' \in V$. Hence, using Theorem~\ref{thm:PhaseStability} we see that the stability of this trivial phase-locked solution can be determined by examining the directed graph induced by the coupling matrix $K$. Moreover, if $K$ can be shown to satisfy the graph-theoretic hypotheses of Theorem~\ref{thm:AsymDecay}, one may use Theorem~\ref{thm:PhaseStability} to infer local asymptotic stability of the trivial phase-locked solution with respect to perturbations in $\ell^1(V)$.

In particular, one can simply define the infinite matrices $K_\mathrm{sym}$ and $K_\mathrm{skew}$ by 
\[
	\begin{split}
		K_\mathrm{sym} &= \frac{1}{2}[K + K^T], \\
		K_\mathrm{skew} &= \frac{1}{2}[K - K^T],	
	\end{split}
\] 
where $K^T = [k_{v',v}]_{v,v'\in V}$ is the formal transpose of the infinite matrix $K$. The entries of $K_\mathrm{sym}$ are exactly the weights of the associated symmetric graph, and hence to satisfy Hypothesis~\ref{hyp:SymmetricWeights} one must first check that the elements of $K_\mathrm{sym}$ are both nonnegative and uniformly bounded above. Furthermore, it must be so that each row and column contains only finitely many nonzero entries, which this number of nonzero entries is uniformly bounded above over all rows and columns. Checking that the symmetric graph defined by $K_\mathrm{sym}$ satisfies $\Delta$, PI, and VG(d) for $d \geq 2$ can be followed as in \cite[Section~6]{Bramburger}. Finally, to satisfy Hypothesis~\ref{hyp:AsymWeights} we must have that the $\ell^1$ norm of the entries of $K_\mathrm{skew}$ is finite.

\section*{Acknowledgements} 

This work is supported by an NSERC PDF held at Brown University.

\end{document}